\documentclass[12pt]{article}

\newtheorem{lemma}{Lemma}
\newtheorem{theorem}{Theorem}
\newtheorem{remark}{Remark}
\newtheorem{corollary}{Corollary}
\newtheorem{example}{Example}

\newtheorem{definition}{Definition}

\def \proof{\medskip \par \noindent {\bf Proof:} \ \ }

\begin{document}
\bibliographystyle{plain}

\thispagestyle{empty}
\setcounter{page}{0}

{\LARGE Guszt\'av MORVAI and Benjamin WEISS: }

\vspace {1cm}

{\Large  On Estimating  the Memory for Finitarily Markovian Processes.}

\vspace {1cm}

{\Large   Ann. Inst. H. Poincar\'e Probab. Statist.  43  (2007),  no. 1, 15--30. }

\vspace {1cm}

\begin{abstract}
Finitarily Markovian processes are those  processes
  $\{X_n\}_{n=-\infty}^{\infty}$ for which
there is a finite $K$ ($ K=K( \{X_n\}_{n=-\infty}^0$) such that the
conditional distribution of $X_1$ given the entire past is equal
to the conditional distribution of $X_1$ given only
   $\{X_n\}_{n=1-K}^0$. The least such value of $K$ is called
   the memory length.
   We give a rather complete analysis of  the problems
     of universally estimating the least such value of $K$, both in the
     backward sense that we have just described and in the forward sense,
     where one observes successive values of  $\{X_n\}$ for  $n \geq 0$
     and asks for the least value $K$ such that the conditional
     distribution of $X_{n+1}$ 
given $\{X_i\}_{i=n-K+1}^n$
     is the same as the conditional distribution of $X_{n+1}$
     given $\{X_i\}_{i=-\infty}^n$. We allow for finite or
     countably infinite alphabet size.

\bigskip

 Les processus Markoviens finitaires sont des processus
  $\{X_n\}_{n=-\infty}^{\infty}$ pour lesquels
il existe un entier  $K$ fini  ($ K=K( \{X_n\}_{n=-\infty}^0$) tel que
la distribution conditionnelle de  $X_1$ etant donn\'e tout le pass\'e
soit \'egale \`a la distribution conditionnelle de  $X_1$ etant donn\'e
seulement
   $\{X_n\}_{n=1-K}^0$. La plus petite valeur d'un tel  $K$ est appel\'ee
   la longueur de la m\'emoire.
   Nous donnons une analyse compl\`ete du probl\`eme de l'estimation de
   la plus petite de ces valeurs de  $K$, aussi bien en remontant dans
   le pass\'e qu'en allant vers le futur, c'est \`a dire quand on observe
   les valeurs successives de  $\{X_n\}$ pour  $n \geq 0$ et
     qu'on recherche la plus petite valeur de $K$ telle que la distribution
     conditionnelle de $X_{n+1}$ etant donn\'e $\{X_i\}_{i=n-K+1}^n$
     soit la m\^eme que la distribution conditionnelle de $X_{n+1}$
     etant donn\'e $\{X_i\}_{i=-\infty}^n$. La taille des alphabets peut
     etre choisie finie ou infinie.
\end{abstract}

\pagebreak

\section{Introduction}

\smallskip
\noindent

  An important class of stationary ergodic processes that greatly
  extends the finite order Markov chains is the
  finitarily Markovian class. Informally, these are those processes
  $\{X_n\}_{n=-\infty}^{\infty}$ for which
there is a finite $K$ (that depends on the past $\{X_n\}, n \leq
0$) such that the conditional distribution of $X_1$ given the entire
past is equal to the conditional distribution of $X_1$ given only
   $\{X_n\}, K < n \leq 0$.
      When the process is a Markov chain of order
     $L$ then one  can simply take $K=L$ independent of the values that
     the process takes. However, even for such Markov chains, quite often
     a smaller value may exist for certain realizations of the process.
     Our main goal here is to give a rather complete analysis of  the problems
     of universally estimating the least such value of $K$, both in the
     backward sense that we have just described and in the forward sense,
     where one observes successive values of  $\{X_n\} , n \geq 0$.
  For the case of finite alphabet finite order Markov chains similar questions
  have been studied by B\"{u}hlman and Wyner in \cite{BW99}. However, the fact
  that we want to treat countable alphabets complicates matters significantly.
  The point is that while finite alphabet Markov chains have exponential
  rates of convergence of empirical distributions, for countable alphabet
  Markov chains no universal rates are available at all.

  We encountered this problem in \cite{MW05} where we gave a universal
  estimator for the order of a Markov chain on a countable state space,
  and some of the techniques that we use here have their origin in that
  paper.
   Before describing our results in more detail let us define more precisely
   the class of processes that we are considering.
\smallskip
\noindent
First let us fix the notation.
Let $\{X_n\}_{n=-\infty}^{\infty}$ be a stationary and ergodic time series taking values from a
discrete (finite or countably infinite)  alphabet
${\cal X}$. (Note that all stationary time series $\{X_n\}_{n=0}^{\infty}$
can be thought to be a
two sided time series, that is, $\{X_n\}_{n=-\infty}^{\infty}$. )
For notational convenience, let $X_m^n=(X_m,\dots,X_n)$,
where $m\le n$. Note that if $m>n$ then $X_m^n$ is the empty string.

\bigskip
\noindent
For  convenience let
  $p(x_{-k}^0)$ and $p(y|x^0_{-k})$ denote the  distribution $P(X_{-k}^{0}=x_{-k}^{0})$ and
the conditional distribution $P(X_1=y|X^0_{-k}=x^0_{-k})$, respectively.

\bigskip
\noindent
\begin{definition} For a stationary  time series $\{X_n\}$ the (random) length
 $K(X^0_{-\infty})$
 of the memory of the sample path
 $X^0_{-\infty}$
is the smallest possible $0\le K<\infty$ such that
for all $i\ge 1$, all $y\in {\cal X}$, all $z^{-K}_{-K-i+1}\in {\cal X}^{i}$
$$
p(y|X^0_{-K+1})=p(y|z^{-K}_{-K-i+1},X^0_{-K+1})
$$
provided $p(z^{-K}_{-K-i+1},X^0_{-K+1},y)>0$,
and $K(X^0_{-\infty})=\infty$ if there is no such $K$.
\end{definition}

\bigskip
\noindent
\begin{definition}
The stationary time series $\{X_n\}$ is said to be finitarily Markovian if
$K(X^0_{-\infty})$ is finite (though not necessarily bounded) almost surely.
\end{definition}

   This class includes of course all finite order Markov chains but also many
   other processes such as the finitarily determined processes of Kalikow, Katznelson and
   Weiss \cite{KKW92},
   which serve to represent all isomorphism classes of zero entropy processes.
   For some concrete examples that are not Markovian consider the following example:

   \begin{example}
\rm{
Let $\{M_n\}$ be any stationary and ergodic first order Markov chain with
finite or countably infinite
state space $S$.
Let $s\in S$ be an arbitrary state with $P(M_1=s)>0$. Now let $X_n=I_{\{M_n=s\}}$.
By Shields \cite{Sh96} Chapter I.2.c.1, the binary time series $\{X_n\}$
is  stationary and ergodic.
It is also finitarily Markovian. Indeed, the conditional probability
$P(X_1=1|X^0_{-\infty})$
does not depend on values
beyond the first (going backwards) occurrence of one  in $X^0_{-\infty}$
which identifies the first (going backwards) occurrence of state $s$ in the
Markov chain $\{M_n\}$. 
The resulting time series $\{X_n\}$ is not a Markov chain of any order in general.
Indeed, consider the Markov chain $\{M_n\}$ with state space $S=\{0,1,2\}$ and
transition probabilities
$P(M_2=1|M_1=0)=P(M_2=2|M_1=1)=1$, $P(M_2=0|M_1=2)=P(M_2=1|M_1=2)=0.5$. 
This yields
a stationary and ergodic Markov chain $\{M_n\}$, cf. 
Example I.2.8 in Shields \cite{Sh96}.
Clearly, the resulting time series  $X_n=I_{\{M_n=0\}}$ will not be 
Markov of any order.
The conditional probability
$P(X_1=0|X^0_{-\infty})$ depends on whether until the first (going backwards)
occurrence of one
you see
even or odd number of zeros.
These examples include all stationary and ergodic binary renewal processes with finite expected
inter-arrival times, a basic class for many applications.
(A stationary and ergodic binary renewal process is defined as
a stationary and ergodic binary process such that the times between occurrences of
ones  are independent and identically distributed with finite expectation, cf.
Chapter I.2.c.1 in Shields \cite{Sh96}).
}
\end{example}
We note that Morvai and Weiss \cite{B-MW05} proved that there is no classification rule
for discriminating the class of finitarily Markovian processes from other ergodic
processes.

  For the finitarily Markovian processes an important notion is that of a
  \textbf{memory word} which is defined as follows.

\begin{definition}
We say that $w^0_{-k+1}$ is a memory word if $p(w^0_{-k+1})>0$ and 
for all $i\ge 1$, all $y\in {\cal X}$, all $z^{-k}_{-k-i+1}\in {\cal X}^{i}$
$$
p(y|w^0_{-k+1})=p(y|z^{-k}_{-k-i+1},w^0_{-k+1})
$$
provided $p(z^{-k}_{-k-i+1},w^0_{-k+1},y)>0$.
\end{definition}

\noindent
Define the set ${\cal W}_k$ of those memory words $w^0_{-k+1}$ with length $k$, that is,
$$
{\cal W}_k=\{w^0_{-k+1}\in {\cal X}^k: \ w^0_{-k+1} \ \mbox{is a memory word}\}.
$$
 Our first result is a solution of  the
 backward estimation problem, namely determining the value of $K(X^0_{-\infty})$
from  observations of increasing length of the data segments $X_{-n}^0$.  We will
give in the next section a universal consistent estimator  which will converge
almost surely to the memory length $K(X^0_{-\infty})$ for any ergodic finitarily
Markovian process on a countable state space. The proofs that we give are
pretty explicit and
given some information on the
average length of a memory word and the extent to which the stationary
distribution diffuses over the state space one could extract rates for the convergence
of the estimators from our estimates. We concentrate however, on the more universal
aspects of the problem.

\smallbreak

As is usual in these kinds
of questions , the problem of forward estimation, namely trying to determine
$K(X^n_{-\infty})$ from successive observations of $X_0^n$ is  more difficult.
  The stationarity means that results in probability can be carried over
  automatically. However, almost sure results present serious problems.
For example, while Ornstein in \cite{Orn78} (cf. Morvai et. al.  \cite{MYGY96} also)  
showed that there is a universal consistent
estimator for the conditional probability of $X_1$ 
given $X_{-\infty}^0$
based on successive observations of the past,
  Bailey
  \cite{Bailey76} showed that one simply cannot
   estimate the forward conditional probabilities in  a similar universal way.
   One can obtain results modulo a zero density set of moments,
   but if one wants to be sure
   that when one is giving an estimate that eventually the estimate
   converges one is forced to resort to estimating along a sequence of
   stopping times (cf Morvai \cite{Mo00}, Morvai and Weiss \cite{MW03}, \cite{T-MW04}, 
\cite{TSP-MW05}).
   For some more results in this circle
of ideas of what can be learned about processes by
forward observations see Ornstein and Weiss \cite{OW90},
Dembo and Peres \cite{DP-1994}, Nobel \cite{N-1999}, and Csisz\'ar 
\cite{Cs02}.

Recently in Csisz\'ar and Talata \cite{CsT} the authors define a finite context to be a memory word $w$ of minimal length, that is, 
no proper suffix of $w$ is a memory word. An infinite context for a process is an infinite string with all finite suffix 
having positive probability but none of them being a memory word. They treat there the problem of estimating 
the entire context tree in case the size of the alphabet is finite. 
For a bounded depth context tree, the process is Markovian, 
while for an unbounded depth context tree the universal pointwise consistency result there is obtained only for the truncated trees 
which are again finite in size. This is in contrast to our results which deal with infinite alphabet size and consistency in estimating 
memory words of arbitrary length. This is what forces us to consider estimating at specially chosen times.

In the succeeding two sections $\S 3,4 $ we will present two such schemes
which depend upon a positive parameter $\epsilon$, and
we guarantee that sequence  of times along which the estimates are being 
given
have density at least $1- \epsilon$. The purpose of the next two sections
is to show that this result is sharp in that
the $\epsilon$ cannot be removed even in more restricted classes of 
processes.
In $\S 5$ we show that you cannot achieve density one in forward estimation
of the memory in the class of Markov chains on countable alphabets, while
in $\S 6$ we prove a similar negative result for binary valued finitarily
Markovian processes.

  The last part of the paper is devoted to seeing how this memory length
  estimation can be applied to estimating conditional probabilities. In $\S 7$
  we do this for finitarily Markovian processes along a sequence of stopping
  times which achieve density $1 - \epsilon$. We do not know if the $\epsilon$
  can be dropped in this case for the estimation of conditional probabilities.

  We can dispense with $\epsilon$
  in the Markovian case. In $\S 8$  we use an earlier
  result of ours on a universal estimator for the order of a
  finite order Markov chain on a countable alphabet in order to estimate the conditional
  probabilities along a sequence of stopping times of density one.

\section{Backward Estimation of the Memory Length for Finitarily Markovian Processes}

\label{chpbackward}

\bigskip
\noindent
In order to estimate $K({X}^0_{-\infty})$ we need to define some explicit statistics.
The first is a  measurement of the failure  of ${w}^0_{-k+1}$ to be a memory
word.

\smallskip
\noindent
For ${w}^0_{-k+1}$ of positive probability define
\begin{eqnarray*}
\lefteqn{
\Delta_k({w}^0_{-k+1})=}\\
&& \sup_{1\le i}
\sup_{\{ z^{-k}_{-k-i+1}\in {\cal X}^i, x\in {\cal X}  :
p(z^{-k}_{-k-i+1},{w}^0_{-k+1},x)>0 \} }
\left| p(x| {w}^0_{-k+1})- p(x|z^{-k}_{-k-i+1} ,{w}^0_{-k+1})\right|.
\end{eqnarray*}

\noindent
Clearly this will vanish precisely when ${w}^0_{-k+1}$ is a memory word. We
need to define an empirical version of this based on the observation
of a finite data segment $X_{-n}^{0}$. To this end first
define the empirical version of the conditional probability as
$$
{\hat p}_n(x|w^0_{-k+1})=
{ \#\{-n+k-1 \le t\le -1: X^{t+1}_{t-k+1}=({w}^0_{-k+1},x)\}\over
\#\{-n+k-1 \le t\le -1: X^t_{t-k+1}={w}^0_{-k+1}\}} .
$$
These empirical distributions, as well as the
sets we are about to introduce are functions of $X^0_{-n}$,
but we suppress the dependence to keep the notation
manageable.

\noindent
For a fixed  $0<\gamma<1$  let ${\cal L}_{k}^n$ denote the
set of strings with length $k+1$  which appear
more than $n^{1-\gamma}$ times in $X_{-n}^{0}$. That is,
$${\cal L}_{k}^n=\{x^0_{-k}\in {\cal X}^{k+1}:
\#\{-n+k\le t\le 0: X^t_{t-k}=x^0_{-k}\} > n^{1-\gamma}\}.$$

 Finally,  define the empirical version of $\Delta_k$ as follows:
$$
{\hat \Delta}^n_k({w}^0_{-k+1})=\max_{1\le i \le n}
\max_{(z^{-k}_{-k-i+1},{w}^0_{-k+1},x)\in {\cal L}^n_{k+i} }
  \left|
{\hat p}_n(x|w^0_{-k+1})-
{\hat p}_n(x|z^{-k}_{-k-i+1},{w}^0_{-k+1})\right|
$$

Let us  agree by convention that if the smallest of the sets
over which we are maximizing is empty then $\hat  \Delta^n_k = 0$.
Observe, that by ergodicity, the ergodic
theorem implies that almost surely the empirical distributions $\hat p$ converge
to the true distributions $p$ and so
 for any $w^0_{-k+1}\in {\cal X}^k$,

$$
\liminf_{n\to\infty}{\hat \Delta}^n_k(w^0_{-k+1})\ge \Delta_k(w^0_{-k+1})
 \ \ \mbox{almost surely.}
$$
\bigskip
\noindent
With this in hand we can give   a test for $w^0_{-k+1}$ to be a memory word.
Let $0< \beta <{1-\gamma \over 2}$ be  arbitrary.
Let $NTEST_n(w^0_{-k+1})=YES$ if ${\hat \Delta }^n_{k}(w^0_{-k+1})\le n^{-\beta}$ and $NO$ otherwise.
Note that $NTEST_n$ depends on $X^0_{-n}$.

\begin{theorem}
\label{ntestthm}
Eventually almost surely,
$NTEST_n(w^0_{-k+1})=YES$ if and only if $w^0_{-k+1}$ is a memory word.
\end{theorem}

\bigskip
\noindent
We define an estimate $\chi_n$ for $K({X}^0_{-\infty})$  from samples
$X_{-n}^{0}$ as follows.
Set $\chi_0=0$, and for $n\ge 1$
let  $\chi_n$ be the smallest
$0\le k< n $ such that
$NTEST_n(X^0_{-k+1})=YES$ if there is such and $n$ otherwise.

\begin{theorem}
\label{thmkbackcons}
$\chi_n=K({ X}^0_{-\infty})$ eventually almost surely.
\end{theorem}

\bigskip
\noindent
In order to prove these theorems we need some lemmas.
The first is a variant of the simple fact that the states $u_i$ that
follow the successive occurrences of a fixed memory word $w$
are independent and identically distributed random variables.
We cannot use such a naive version because we are dealing with a countable
alphabet, and thus even the collection  of memory words of a fixed
length is infinite. In order to cut down to a manageable set we would like
to consider only those words that appear in the sample $X_{-n}^0$, but
now the independence becomes a little subtler. This is the reason
for the rather forbidding looking formulas in the proof of the next lemma.
What we do is fix a location $(l-k,l]$ in the index set and then fix
a memory word $w_{-k+1}^0$ that occurs there together with a particular state $x$
that follows it. The random times $l+\lambda^{+}_{\cdot}$ and $l-\lambda^{-}_{\cdot}$ are the other
occurrences of this memory word in the process. Here is the
formal definition.
Set $\lambda_{l,k,0}^+=0$,  $\lambda_{l,k,0}^-=0$ and define
\begin{equation}\label{poslambda}
\lambda_{l,k,i}^+=
\lambda_{l,k,i-1}^+ +
\min\{t>0: X^{l+\lambda_{l,k,i-1}^+ +t}_{l+\lambda_{l,k,i-1}^+-k+1 +t}=
X^{l+\lambda_{l,k,i-1}^+}_{l+\lambda_{l,k,i-1}^+ -k+1}\}
\end{equation}
and
\begin{equation}\label{neglambda}
\lambda_{l,k,i}^-=
\lambda_{l,k,i-1}^- +
\min\{t>0: X^{l-\lambda_{l,k,i-1}^- -t}_{l-\lambda_{l,k,i-1}^- -k+1-t}=
X^{l- \lambda_{l,k,i-1}^-}_{l- \lambda_{l,k,i-1}^- -k+1}\}
\end{equation}

\begin{lemma} \label{iidlemma} Assume $w^0_{-k+1}$ is a memory word and $x$ is a letter. Then for any $i,j\ge 1$, 
$$X_{l-\lambda_{l,k,i}^-+1},\dots,X_{l-\lambda_{l,k,1}^-+1},
X_{l+\lambda_{l,k,1}^+ +1},\dots,X_{l+\lambda_{l,k,j}^+ +1} $$
are conditionally  independent and identically distributed random 
variables given $X_{l-k+1}^l=w^0_{-k+1},X_{l+1}=x$, where the identical 
distribution is $p(\cdot|w^0_{-k+1})$.
 \end{lemma}
\proof
   Fix the values $z_{-i},\dots, z_{-1},u_1,\dots,u_j$ and $x$ in the alphabet and calculate

\begin{eqnarray*}
\lefteqn{
P(X_{l-\lambda_{l,k,i}^- +1}=z_{-i},\dots,X_{l-\lambda_{l,k,1}^- +1}=z_{-1},}\\
&&X_{l+\lambda_{l,k,1}^+ +1}=u_1,\dots, X_{l+\lambda_{l,k,j}^+ +1}=u_j|X^l_{l-k+1}=w^0_{-k+1},X_{l+1}=x)\\
&=&
P(X_{l-\lambda_{l,k,i}^- +1}=z_{-i},\dots,X_{l-\lambda_{l,k,1}^- +1}=z_{-1},
X_{l+\lambda_{l,k,1}^+ +1}=u_1,\dots,X_{l+\lambda_{l,k,j}^+ +1}=u_j,\\
&&X^l_{l-k+1}=w^0_{-k+1}X_{l+1}=x)
/ P(X^l_{l-k+1}=w^0_{-k+1},X_{l+1}=x)
\end{eqnarray*}
  In order to be able to use the fact that $w^0_{-k+1}$ is a memory word we
  will shift back to the first occurrence at $l - \lambda_{l,k,i}^-$ and use the stationarity.
\begin{eqnarray*}
\lefteqn{
P(X_{l-\lambda_{l,k,i}^- +1}=z_{-i},\dots,X_{l-\lambda_{l,k,1}^- +1}=z_{-1},
X_{l+\lambda_{l,k,1}^+ +1}=u_1.\dots,X_{l+\lambda_{l,k,j}^+
+1}=u_j}\\
&&,X^l_{l-k+1}=w^0_{-k+1},X_{l+1}=x,\lambda_{l,k,i}^- =t)\\
&=&
P(X_{l-t+\lambda_{l-t,k,i-i}^+ +1}=z_{-i},\dots,X_{l-t+\lambda_{l-t,k,i-1}^+ +1}=z_{-1},\\
&&
X_{l-t+\lambda_{l-t,k,i+1}^+ +1}=u_1,\dots,X_{l-t+\lambda_{l-t,k,i+j}^+
+1}=u_j,\\
&&
X^{l-t+\lambda_{l-t,k,i}^+}_{l-t+\lambda_{l-t,k,i}^+ -k+1}=w^0_{-k+1},
X_{l-t+\lambda_{l-t,k,i}^+ +1}=x,\lambda_{l-t,k,i}^+ =t)\\
&=&
P(T^{-t} \{X_{l-t+\lambda_{l-t,k,i-i}^+ +1}=z_{-i},\dots,X_{l-t+\lambda_{l-t,k,i-1}^+ 
+1}=z_{-1},\\
&&
X_{l-t+\lambda_{l-t,k,i+1}^+ +1}=u_1,\dots,X_{l-t+\lambda_{l-t,k,i+j}^+
+1}=u_j,X^{l-t+\lambda_{l-t,k,i}^+}_{l-t+\lambda_{l-t,k,i}^+ -k+1}=w^0_{-k+1},\\
&&X_{l-t+\lambda_{l-t,k,i}^+ +1}=x,\lambda_{l-t,k,i}^+
=t\})\\
&=&
P(X_{l+\lambda_{l,k,0}^+ +1}=z_{-i},\dots,X_{l+\lambda_{l,k,i-1}^+ +1}=z_{-1},
X_{l+\lambda_{l,k,i+1}^+ +1}=u_1,\dots,X_{l+\lambda_{l,k,i+j}^+
+1}=u_j,\\
&& X^{l+\lambda_{l,k,i}^+}_{l+\lambda_{l,k,i}^+ 
-k+1}=w^0_{-k+1},X_{l+\lambda_{l,k,i}^+ +1}=x,\lambda_{l,k,i}^+  =t)
\end{eqnarray*}
Summing over $t$,
\begin{eqnarray}
\lefteqn{\nonumber P(X_{l-\lambda_{l,k,i}^- 
+1}=z_{-i},\dots,X_{l-\lambda_{l,k,1}^- +1}=z_{-1},
X_{l+\lambda_{l,k,1}^+ +1}=u_1,\dots,X_{l+\lambda_{l,k,j}^+ +1}=u_j,}\\
&&\nonumber X^l_{l-k+1}=w^0_{-k+1}X_{l+1}=x)\\
&=&\nonumber
P(X_{l+\lambda_{l,k,0}^+ +1}=z_{-i},\dots X_{l+\lambda_{l,k,i-1}^+ +1}=z_{-1},
X_{l+\lambda_{l,k,i+1}^+ +1}=u_1,\dots,X_{l+\lambda_{l,k,i+j}^+
+1}=u_j,\\
&& 
\label{keybel}
X^{l+\lambda_{l,k,i}^+}_{l+\lambda_{l,k,i}^+ 
-k+1}=w^0_{-k+1},X_{l+\lambda_{l,k,i}^++1}=x).
\end{eqnarray}
Now telescoping the right hand side we get
\begin{eqnarray*}
\lefteqn{P(X_{l-\lambda_{l,k,i}^- +1}=z_{-i},\dots,X_{l-\lambda_{l,k,1}^- +1}=z_{-1},}\\
&&
X_{l+\lambda_{l,k,1}^+ +1}=u_1,\dots,X_{l+\lambda_{l,k,j}^+ +1}=u_j|X^l_{l-k+1}=w^0_{-k+1},X_{l+1}=x) \\
&=& P(X^0_{-k+1}=w^0_{-k+1})\prod_{h=1}^i 
P(X_1=z_{-h}|X^0_{-k+1}=w^0_{-k+1})
P(X_1=x|X^0_{-k+1}=w^0_{-k+1})\\
&\cdot& { \prod_{h=1}^j P(X_1=u_h|X^0_{-k+1}=w^0_{-k+1})\over
P(X^0_{-k+1}=w^0_{-k+1}) P(X_1=x|X^0_{-k+1}=w^0_{-k+1})}\\
&=& \prod_{h=1}^i P(X_1=z_{-h}|X^0_{-k+1}=w^0_{-k+1}) \prod_{h=1}^j 
P(X_1=u_h|X^0_{-k+1}=w^0_{-k+1}).
\end{eqnarray*}
We have to prove that
$$
P(X_1=z_{-h}|X^0_{-k+1}=w^0_{-k+1})=
P(X_{l-\lambda_{l,k,h}^- +1}=z_{-h}|X^l_{l-k+1}=w^0_{-k+1},X_{l+1}=x).
$$
Indeed, by ~(\ref{keybel}) and stationarity,
\begin{eqnarray*}
\lefteqn{ 
P(X_{l-\lambda_{l,k,h}^-+1}=z_{-h}|X^l_{l-k+1}=w^0_{-k+1},X_{l+1}=x)}\\
&=& {P(X_{l-\lambda_{l,k,h}^- 
+1}=z_{-h},X^{l-\lambda_{l,k,h}^-}_{l-\lambda_{l,k,h}^- 
-k+1}=w^0_{-k+1},X_{l+1}=x)\over P(X^l_{l-k+1}=w^0_{-k+1},X_{l+1}=x)}\\
&=& P(X^0_{-k+1}=w^0_{-k+1})
P(X_1=z_{-h}|X^0_{-k+1}=w^0_{-k+1})\\
&\cdot& {
P(X_{l+1}=x|X^l_{l-k+1}=w^0_{-k+1},X_{l-\lambda_{l,k,h}^- +1}=z_{-h})\over
P(X^0_{-k+1}=w^0_{-k+1})
P(X_{l+1}=x|X^l_{l-k+1}=w^0_{-k+1})}\\
&=&
P(X_1=z_{-h}|X^0_{-k+1}=w^0_{-k+1}).
\end{eqnarray*}
The proof of Lemma~\ref{iidlemma} is complete.

\begin{lemma} \label{keylemma1}
\begin{eqnarray*}
\lefteqn{
P\left( \ \mbox{For some } 0\le k<n, -n+k-1\le l\le -1: X_{l-k+1}^{l+1}\in {\cal L}^n_{k+1}, K(X^l_{-\infty})\le k,
\right.
}\\
&&\left. \left|{\hat p}_n(X_{l+1}|X_{l-k+1}^l)-p(X_{l+1}|X_{l-k+1}^l)\right|> n^{-\beta}\right)\\
&\le&
n^2 \sum_{h= \lfloor n^{1-\gamma}\rfloor }^{\infty} h 2 e^{-2 n^{-2\beta}h}.
\end{eqnarray*}
\end{lemma}
\proof
For a given $0\le k<n$, $-n+k-1\le l \le -1$ assume that $X_{l-k+1}^{l+1}=w^0_{-k+1} x$ and $w^0_{-k+1}$ is a memory
word.  Since $w^0_{-k+1}$ is a memory word,
by Lemma~\ref{iidlemma}
and  by Hoeffding's inequality (cf. Hoeffding \cite{Hoeffding63} or Theorem 8.1 of Devroye et. al. \cite{DGYL96}) for
 sums of bounded independent random variables implies
\begin{eqnarray*}
\lefteqn{
P\left( \left|{ \sum_{h=1}^i 1_{\{X_{l-\lambda_{l,k,h}^- +1}=x\}} +
 \sum_{h=1}^j 1_{\{X_{l+\lambda_{l,k,h}^+ +1}=x\}}
\over i+j}
-p(x|w^0_{l-k+1})\right|\right. }\\
&\ge& n^{-\beta}| \left. X_{l-k+1}^{l+1}=w^0_{-k+1}x \right) \le 2e^{-2n^{-2\beta}(i+j)}.
\end{eqnarray*}
Multiplying both sides by
$P(X_{l-k+1}^{l+1}=w^0_{-k+1}x)$ and summing over all possible memory words $w^0_{-k+1}$ 
and $x$
we get that
\begin{eqnarray*}
\lefteqn{
P\left( K(X_{-\infty}^{l})\le k, X_{l-k+1}^{l+1}\in {\cal L}_{k+1}^n,\right. }\\
&&\left.
\left|{ \sum_{h=1}^i 1_{\{X_{l-\lambda_{l,k,h}^- +1}=X_{l+1}\}} +
 \sum_{h=1}^j 1_{\{X_{l+\lambda_{l,k,h}^+ +1}=X_{l+1}\}}
\over i+j}
-p(X_{l+1}|X^l_{l-k+1})\right|
> n^{-\beta} \right)
\\
&\le& 2e^{-2n^{-2\beta}(i+j)}.
\end{eqnarray*}
Summing
 over all pairs $(k,l)$ such that $0\le k<n$ and all $-n+k-1\le l\le -1$
and  over all pairs $(i,j)$ such that
$i\ge 0$, $j\ge 0$, $i+j\ge \lfloor n^{1-\gamma}\rfloor $
we complete the proof of Lemma~\ref{keylemma1}.

\begin{lemma} \label{keylemma}
$$
P( \max_{w^0_{-k+1}\in {\cal W}_k}
{\hat \Delta}^n_k(w^0_{-k+1})>n^{-\beta})
\le  n^3 \sum_{h=\lfloor n^{1-\gamma}\rfloor }^{\infty} h4 e^{- n^{-2\beta}h\over 2}.
$$
\end{lemma}
\proof
\begin{eqnarray*}
\lefteqn{
P( \max_{w^0_{-k+1}\in {\cal W}_k}
 {\hat \Delta}^n_k(w^{0}_{-k+1})> n^{-\beta})}\\
&\le&\sum_{i=1}^{n}
  P( \max_{w^0_{-k+1}\in {\cal W}_k}
  \max_{(z^{-k}_{-k-i+1},{w}^0_{-k+1},x)\in {\cal
L}^n_{k+i} }  \left|
{\hat p}_n(x|{w}^0_{-k+1})  -
{\hat p}_n(x| z^{-k}_{-k-i+1},{w}^0_{-k+1})\right| > n^{-\beta})\\
&\le&\sum_{i=1}^{n}
  P( \max_{w^0_{-k+1}\in {\cal W}_k}
  \max_{(z^{-k}_{-k-i+1},{w}^0_{-k+1},x)\in {\cal
L}^n_{k+i}
}  \left|
{\hat p}_n(x|{w}^0_{-k+1})  -
p(x|{w}^0_{-k+1})\right|  > n^{-\beta}/2)\\
&+&
\sum_{i=1}^{n}
  P( \max_{w^0_{-k+1}\in {\cal W}_k} \max_{(z^{-k}_{-k-i+1},{w}^0_{-k+1},x)\in {\cal
L}^n_{k+i} }  \left|
 p(x| z^{-k}_{-k-i+1},{w}^0_{-k+1})-
{\hat p}_n(x| z^{-k}_{-k-i+1},{w}^0_{-k+1})\right| \\
&&
> n^{-\beta}/2)
 \end{eqnarray*}
By Lemma~\ref{keylemma1}, both terms inside the sum can be upperbounded by an exponential,
and summing over $i$ we get the statement and so the proof of Lemma~\ref{keylemma} is complete.

\bigskip
\noindent
{\bf Proof of Theorem~\ref{ntestthm}: }

\noindent
If $w^0_{-k+1}$ is not a memory word, then
there are $z^{-k}_{-k-i+1}$ and $x$ such that
$p(x|w^0_{-k+1})\neq p(x|z^{-k}_{-k-i+1} w^0_{-k+1})$ and
$p(z^{-k}_{-k-i+1} w^0_{-k+1}x)>0$. By ergodicity,
$NTEST_n(w^0_{-k+1})=NO$ eventually almost surely.

\noindent
Assume  $w^0_{-k+1}$ is  a memory word.
We will estimate  the probability of the undesirable  event as follows:
By Lemma~\ref{keylemma},
$$
P({\hat \Delta}^n_k(w^0_{-k+1})> n^{-\beta})
\le
 n^3 \sum_{h=\lfloor n^{1-\gamma}\rfloor }^{\infty} h4 e^{- n^{-2\beta}h \over 2}.
$$
The right hand side is summable provided $2\beta+\gamma<1$ and
the Borel-Cantelli Lemma yields that
$$
P({\hat \Delta}^n_k(w^0_{-k+1})\le n^{-\beta}  eventually)=1
$$
and so $NTEST_n(w^0_{-k+1})=YES$ eventually almost surely.
The proof of Theorem~\ref{ntestthm} is complete.

\bigskip
\noindent
{\bf Proof of Theorem~\ref{thmkbackcons}: }
Since $X^0_{-K(X^0_{-\infty})+1}$ is a memory word and none of its suffixes has this property,
$\chi_n= K(X^0_{-\infty})$ eventually almost surely, by 
Theorem \ref{ntestthm} .
The proof of Theorem~\ref{thmkbackcons} is complete.

\section{Forward Estimation of the Memory Length for Finitarily Markovian Processes}

\noindent
Define $PTEST_n(w^0_{-k+1})(X^n_0)=NTEST_n(w^0_{-k+1})(T^n X^n_0)$ where $T$ is the left shift operator.

\begin{theorem}
\label{ptestthm}
Eventually almost surely,
$PTEST_n(w^0_{-k+1})=YES$ if and only if $w^0_{-k+1}$ is a memory word.
\end{theorem}

\noindent
Define a list of words $\{w(0),w(1),w(2),\dots,w(n),\dots\}$ such that all words of all 
lengths are listed and a word can not
precede its
suffix. Note that $w(0)$ is the empty word.

\noindent
Now define sets of indices $A^i_n$ as follows. Let $A^0_n=\{0,1,\dots,n\}$ and for $i>0$ define
\begin{equation}
A^i_n=\{ |w(i)|-1 \le j\le n: X^j_{j-|w(i)|+1}=w(i)\}.
\end{equation}
Let $\epsilon>0$ be fixed. Define $\theta_n(\epsilon)< n$ to be the minimal $j$ such that
\begin{equation}
{\left|\bigcup_{i\le j: PTEST_n(w(i))=YES} A^i_n\right|\over n+1}\ge 1-\epsilon/2
\end{equation}
and $n$ if no such $j$ exists.
We estimate for the length of the memory of $X^n_{-\infty}$ looking backwards if
$n\in \bigcup_{i\le
\theta_n(\epsilon),PTEST_n(w(i))=YES} A^i_n$.
The set of $n$'s for which this holds will be the set for which we estimate the memory and we denote this set by
${\cal N}$.  Note that
the event $n\in{\cal  N}$ depends only on $X^n_0$, and thus ${\cal N}$ can be thought of as a sequence of
stopping times.

\noindent
We define for $n\in {\cal N}$,
$$
\kappa_n=\min\{i\ge 0: X^n_{n-|w(i)|+1}=w(i), PTEST_n(w(i))=YES \}.
$$
For $n\in {\cal N}$ define
$$
\rho_{n}(X^n_0)=|w(\kappa_n)|.
$$
Note that $\rho_n$, $\theta_n$, $\kappa_n$ and ${\cal N}$  depend on $\epsilon$, however, we will not denote this
dependence
on $\epsilon$ explicitly.

\begin{theorem} \label{postheoremp} Let $\epsilon>0$ be fixed. Then for $n\in {\cal N}$,
\begin{equation}\label{conspartthmp}
\rho_{n}=K(X^{n}_{-\infty}) \ \mbox{eventually almost surely,}
\end{equation}
and
\begin{equation}\label{growthpartthmp}
\liminf_{n\to\infty} { \left| {\cal N}\bigcap \{0,1,\dots,n-1\} \right| \over n}\ge 1-\epsilon.
\end{equation}
For $n\in \cal N$, $X^n_{n-\rho_n+1}$ appears at least $n^{1-\gamma}$ 
times eventually almost surely.
\end{theorem}

\bigskip
\noindent
{\bf Proof of Theorem~\ref{ptestthm}:}
\bigskip
\noindent
Since the proof of  Theorem~\ref{ntestthm}
was based on a Borel-Cantelli lemma, the time shift in defining $PTEST_n$ makes no difference and we literally
copy the proof of Theorem~\ref{ntestthm}.
The proof of Theorem~\ref{ptestthm} is complete.

\bigskip
\noindent
{\bf Proof of Theorem~\ref{postheoremp}:}
\bigskip
\noindent
There is a $N$ large enough such that
$$
P(K(X^0_{-\infty})<N)\ge 1-\epsilon/4.
$$
The sequence $\theta_n$ is bounded along individual sequences of the process with probability one. (This may be seen by first choosing a 
sufficiently large finite set $\{w(0),\dots,w(M)\}$ of memory words so that the probability of seeing at least one of them 
in position zero is greater then $1-\epsilon/4$ and  then applying 
Theorem~\ref{ptestthm} and the ergodic theorem we see that almost surely for all sufficiently large $n$,  $\theta_n\le M$. 
This implies of course that $\theta_n$ is bounded pointwise as claimed.)   
Thus  by Theorem~\ref{ptestthm},
$$
\rho_n=K(X^n_{-\infty}) \
\mbox{provided $X^n_{n-K(X^n_{-\infty})+1}\in {\cal W}_{K(X^n_{-\infty})}\bigcap\{w(0),\dots,w(\theta_n)\}$}
$$
eventually almost surely.We have proved the consistency.
Let $J$ denote the  smallest $j$  such that
$$
\sum_{k=0}^j  P( X^0_{-k+1}\in {\cal W}_k\bigcap \{w(0),\dots,w(j)\})\ge 1-{\epsilon\over 
2}.
$$
It is obvious from the definition above that
$$
\sum_{k=0}^{J-1} P( X^0_{-k+1}\in {\cal W}_k\bigcap \{w(0),\dots,w(J-1)\}) <
1-{\epsilon\over 2}.
$$
Thus $\theta_n\ge J$ eventually almost surely.
Thus
\begin{eqnarray*}
\lefteqn{
\liminf_{n\to\infty} { \left|{\cal N}\bigcap \{0,1,\dots,n\}\right| \over n+1} }\\
&\ge&
\liminf_{n\to\infty} {\left|\bigcup_{i=0}^J A^i_n\right| \over n+1}\ge 1-{\epsilon\over 2}\ \ \mbox{almost surely.}
\end{eqnarray*}
We have proved that ${\cal N}$ has density at least $1-\epsilon/2$.
Since $\theta_n$ is bounded, 
for $n\in {\cal N}$ eventually, $w(\kappa_n)$ appears at least $n^{1-\gamma}$ times.
The proof of Theorem~\ref{postheoremp} is complete

\section{Another Approach to Estimating the Memory Length for Finitarily Markovian Processes}

In the preceding section we made use of the fact that the proof that we gave for the backward memory estimator
was via a rough probability estimate and the Borel-Cantelli lemma. This enabled us to copy it directly for the
forward estimation. In this section we shall show that any successful backward memory estimator can be used to get
the same kind of result.
We will denote by $\chi_n$ some fixed consistent backward estimator for the memory length 
such as the $\chi_n$ of
\S~ \ref{chpbackward}.
To this end, based on the successive forward samples we construct many infinite sample points
of the $X^0_{-\infty}$ process.

\noindent
\bigskip
To construct a sample of the $X^0_{-\infty}$ process from the forward data segment $X_0^n$, we
use the procedure that we used in Morvai and Weiss \cite{PTRF-MW05}. Begin with $X_0$, 
then look for its first recurrence,
i.e. the minimum $t_0>0$ such that $X_{t_0}=X_0$ and then extend $X_0$ to the left by adding $X_{t_0-1}$. Next
look for the first recurrence of $X_{t_0-1} X_{t_0}$, in a position $t_1>t_0$, i.e. $X_{t_1-1}X_{t_1}=X_{t_0-1}X_{t_0}$
and then again extend to the left by adding $X_{t_1-2}$ obtaining $X_{t_1-2}X_{t_1-1}X_{t_1}$ as the first three
symbols of our sample for the backward process. We will denote this by ${\tilde X}^0_{-2}=X^{t_1}_{t_1-2}$. Continuing
in this way, we can develop  from $X^{\infty}_0$ a point ${\tilde X}^0_{-\infty}$ which we shall show has the same
distribution as $X^0_{-\infty}$.
We need to do this starting at each $i\ge 0$. Here are the formulas that accomplish this end.

\bigskip
\noindent
For $i=0,1,\dots$ define auxiliary  stopping times.   Set $\zeta_{-1}(i)=-i$ and  $\zeta_0(i)=0$.
For $n=1,2,\ldots$,  let
\begin{equation}\label{defzeta}
\zeta_n(i)=\zeta_{n-1}(i)+
\min\{t>0 : X_{i+\zeta_{n-1}(i)-(n-1)+t}^{i+\zeta_{n-1}(i)+t}=X_{i+\zeta_{n-1}(i)-(n-1)}^{i+\zeta_{n-1}(i)}\}.
\end{equation}

\smallskip
\noindent
Among other things, using  $\zeta_n(i)$ we can  define  very useful processes
$\{ {\tilde X}_n(i)\}_{n=-\infty}^{0}$
as a function of $X_0^{\infty}$ as follows.

\noindent
Define
\begin{equation}
\label{defprocesses}
{\tilde X}_{-n}(i)=X_{i+\zeta_{n}(i)-n}.
\end{equation}
It is clear that in this way we defined processes
$\{ {\tilde X}_n(i)\}_{n=-\infty}^{0}$. We will see
that the $\{{\tilde X}_n(i)\}_{n=-\infty}^0$ has the same
distribution as the original process, and for now assume that this is so.

\noindent
Let
\begin{equation}
\eta_n(i)=\max\{j\ge -1: i+\zeta_j(i)\le n\}
\end{equation}
Note that $({\tilde X}_{-\eta_n(i)}(i), \dots,{\tilde X}_0(i))$ is
measurable with respect to $X_0^n$.

\noindent
Define $\rho^i_n=\chi_{\eta_n(i)}({\tilde X}^0_{-\infty}(i))$ if $\eta_n(i)\ge 0$ and $\rho^i_n=0$ otherwise.
Note that $\rho^i_n$ is also measurable with respect to $X_0^n$.

\noindent
Define sets of indices $A^i_n$ as follows.
\begin{equation}
A^i_n=\{ \rho^i_n  \le j\le n: X^j_{j-\rho^i_n+1}={\tilde X}^{0}_{-\rho^i_n+1}(i)\}.
\end{equation}
For any fixed $i$, eventually, ${\tilde X}^0_{-\rho_n^i+1}(i)$ is a memory word, so the sets $A_n^i$ are simply
the places where this  fixed word occures.
Let $\epsilon>0$ be fixed. Define $\theta_n(\epsilon)$ to be the minimal $j$ such that
\begin{equation}
{\left|\bigcup_{i\le j} A^i_n\right|\over n}\ge 1-\epsilon/2.
\end{equation}
We estimate for the order of $X^n_{-\infty}$ looking backwards if $n\in \bigcup_{i\le \theta_n(\epsilon)} A^i_n$.
The set of $n$'s for which this holds will be the set for which we estimate the memory and we denote this set by
${\cal N}$.  Note that
the event $n\in{\cal  N}$ depends only on $X^n_0$, and thus ${\cal N}$ can be thought of as a sequence of
stopping times.

\noindent
In case $n\in \bigcup_{i\le \theta_n(\epsilon)} A^i_n$ we define
$$
\kappa_n=\min\{i\ge 0: {\tilde X}^0_{-\rho_n^i+1}(i)=X^{n}_{n-\rho_n^i+1}\}.
$$
Note that $\theta_n$, $\kappa_n$ and ${\cal N}$  depend on $\epsilon$, however, we will not denote this
dependence
on $\epsilon$ explicitly.

\begin{theorem} \label{postheorem} Let $\epsilon>0$ be fixed. Then for $n\in {\cal N}$,
\begin{equation}\label{conspartthm}
\rho_{n}^{\kappa_n}=K(X^{n}_{-\infty}) \ \mbox{eventually almost surely,}
\end{equation}
and
\begin{equation}\label{growthpartthm}
\liminf_{n\to\infty} { \left| {\cal N}\bigcap \{0,1,\dots,n-1\} \right| \over n}\ge 1-\epsilon.
\end{equation}
For $n\in \cal N$, $X^n_{n-\rho_n^{\kappa_n}}$ appears at least 
$n^{1-\gamma}$ times eventually almost surely.
\end{theorem}

\begin{lemma} \label{eqdistrlemma}
For all $i$ the  time series
$\{{\tilde X}_n(i)\}_{n=-\infty}^0$ and
$\{X_n\}_{n=-\infty}^{0}$ have
identical  distribution.
\end{lemma}
\proof
\noindent
For all $k\ge 1$ and
$1\le i\le k$
define
$\hat\zeta^k_0=0$ and
$$
\hat\zeta^k_i=\hat\zeta^k_{i-1}-
\min\{t>0 :
X_{\hat\zeta^k_{i-1}-(k-i)-t}^{\hat\zeta^k_{i-1}-t}
=
{X}_{\hat\zeta^k_{i-1}-(k-i)}^{\hat\zeta^k_{i-1}}\}.
$$
Let $T$ denote the left shift operator,
that is, $(T x^{\infty}_{-\infty})_i=x_{i+1}$. It is easy to see that if  and only if
$\zeta_k(i)(x_{-\infty}^{\infty})=l$ then
${\hat \zeta}^k_k(T^{(i+l)} x_{-\infty}^{\infty})=-l$.
\noindent
Now the statement follows from stationarity and the fact that for $k\ge 0$,
$x^{0}_{-k}\in {\cal X}^{k+1}$,  $l\ge 0$,
\begin{equation}\label{shiftequation}
T^{i+l} \{X^{i+\zeta_k(i)}_{i+\zeta_k(i)-k}=x^0_{-k},\zeta_k(i)=l\} =
\{ X^{0}_{-k}=x^{0}_{-k},{\hat \zeta}^k_k(X^0_{-\infty})=-l\}.
\end{equation}
The proof of Lemma~\ref{eqdistrlemma} is complete.

\begin{lemma} \label{appearlemma}
If $P(X_0^n=w_0^n)>0$  for the string $w^n_0$   then 
almost surely,
\begin{equation}
{\tilde X}^0_{-n}(i)=w^n_0 \ \mbox{for some i}.
\end{equation}
\end{lemma}
\proof
Let $t$ denote the $n+1$-th occurrence of the string $w_0^n$ in $X^{\infty}_{0}$.
It is easy to see that there must be a $0\le i\le t$ such that
$$
X^{i+\zeta_k(i)}_{i+\zeta_k(i)-k}=X^t_{t-k} \ \mbox{for $k=0,1,\dots, n$}
$$
and so
$$
{\tilde X}^0_{-n}(i)=w_0^n.
$$
The proof of Lemma~\ref{appearlemma} is complete.

\bigskip
\noindent
{\bf Proof of Theorem~\ref{postheorem}}
\bigskip
\noindent
There is a $N$ large enough such that
$$
P(K(X^0_{-\infty})<N)\ge 1-\epsilon/4.
$$
Then by Lemma~\ref{appearlemma} and ergodicity, $\theta_n$ is a bounded sequence
(cf.  the proof of Theorem~\ref{postheoremp}). 
By Lemma~\ref{eqdistrlemma} and
Theorem~\ref{thmkbackcons}
$$
\rho_n^i=K({\tilde X}^0_{-\infty}(i)) \ \mbox{for all $i=1,\dots,\theta_n$}
$$
eventually almost surely.
We have proved (\ref{conspartthm}).
We have to prove (\ref{growthpartthm}).
Let $J$ denote the  smallest $j$  such that
$$
\sum_{i=0}^j p({\tilde X}^0_{-K({\tilde X}^0_{-\infty}(i))}(i))\ge
1-{\epsilon\over 2}.
$$
It is obvious from the definition above that
$$
\sum_{i=0}^{J-1} p({\tilde X}^0_{-K({\tilde X}^0_{-\infty}(i))}(i))<
1-{\epsilon\over 2}.
$$
Thus
$\theta_n\ge J$ eventually almost surely.
Thus
\begin{eqnarray*}
\lefteqn{
\liminf_{n\to\infty} { \left|{\cal N}\bigcap \{0,1,\dots,n\}\right| \over n+1} }\\
&\ge&
\liminf_{n\to\infty} {\left|\bigcup_{i=0}^J A^i_n\right| \over n+1}\ge 1-{\epsilon\over 2}\ \ \mbox{almost surely.}
\end{eqnarray*}
The proof of Theorem~\ref{postheorem} is complete.

\section{Memory Estimation for  Markov Processes}

   In this section we shall examine how well can one estimate
   the local memory length for finite order Markov chains. In the
   case of finite alphabets this can be done with stopping times
   that eventually cover all time epochs.
(Indeed,
assume $\{X_n\}$ is a Markov chain taking values from a finite set.
Assume $ORDEST_n$ estimates the order in a pointwise sense from 
data $X^n_0$ e.g. as in Csisz\'ar and  Shields
\cite{CsSh00} or in Morvai and Weiss \cite{MW05}. Then let
$$
\rho_n=\min\{0\le t\le ORDEST_n: \ PTEST_n(X^n_{n-t+1})=YES\}
$$
if there is such $t$ and $0$ otherwise.
Since $ORDEST_n$ eventually gives the right order and there are finitelly many possible strings with length not greater
than the order thus $\rho_n=K(X^n_{-\infty})$ eventually almost surely by Theorem~\ref{ptestthm}.)

However, as soon
   as one goes to a countable alphabet, even if the order is known
   to be two and we are just trying to decide whether the $X_n$ alone
   is a memory word or not, there is no sequence of stopping
   times which is guaranteed to succeed eventually and whose density
   is one. This shows that the $\epsilon$ in the preceding sections
   cannot be eliminated.

\begin{theorem} \label{couplingtheorem}
 There are no strictly increasing sequence of stopping times $\{\lambda_n\}$
and estimators $\{h_n(X_{0},\dots,X_{\lambda_n})\}$
taking the values one and two,
such that for all countable alphabet  Markov chains of order two:
 $$\lim_{n\to\infty} {\lambda_n\over n}=1$$
 and

$$
   \lim_{n\rightarrow \infty} |h_n(X_0,\dots,X_{\lambda_n}) - K(X^{\lambda_n}_{0}) | = 0
\   \  \mbox{with probability one.}
 $$
\end{theorem}

    To prove the theorem we will assume that such a pair of stopping times
    and estimators exist and construct a Markov chain $\{X_n\}$ of order two for
    which  they fail.
   The  Markov chain of order two that we construct will have for its
   state space the nonnegative integers ${\mathbf N}$ , and it
    will be a perturbation of the
   1-step Markov chain $Z_n$ defined by the following formulae:

$$ P_{s,s+r} = 2^{-r-1}\ \ \mbox{for all} \ r \geq 1,$$
$$P_{s,s} = 2^{-s-1}$$and
$$ P_{s,j} = 2^{-j-2}\ \ \mbox{for all} \ 0\leq j < s.$$

  Notice that from any state $s>0$ there is a fixed probability of $1 
\over 4$ of
  going to $0$. Also there is a strictly positive probability of going from
  any state to any other. These properties ensure that there is a finite
  stationary measure. The ultimate chain will preserve most of these conditional
  probabilities, with the difference depending on a sequence of integers $t_k  >> k$
  which will be defined later. The perturbed chain $\{X_n\}$ will have
   the  same transition probabilities as the original chain $\{Z_n\}$ 
for $X_n$ given  $(X_{n-2},X_{n-1})$ when the latter, 
$(X_{n-2},X_{n-1})$, equals  
   any pair $(t,s)$ with the exception of  $(t_k,k)$ 
for $k\geq0$.
   In that
  case we will modify the probability of the transitions to $k$ and $k+1$ by
  interchanging the values of $P_{k,k}$ and $P_{k,k+1}$.
  As soon as the first change is made $0$ ceases to be a memory word and therefore the order of the new chain is two. Eventually all  
  singletons  cease to be memory words. 

    The $t_k$'s will be chosen inductively in a fashion depending on the
    purported sequence of stopping times and estimators that we are trying
    to show cannot exist. At the k-th stage we will have only made these
    changes up to   k.
    Let us denote the Markov process of order $2$  that this defines
    by $\{Y^{(k)}_n\}$. More explicitly the process $\{Y^{(k)}_n\}$ is 
defined as follows. It has transition probabilities given by:
 for all $j\le k$
$$
P(Y^{(k)}_{n+2}=j|Y^{(k)}_n=t_j,Y^{(k)}_{n+1}=j)=P(Z_{n+2}=j+1|Z_{n+1}=j),
$$ 
$$
P(Y^{(k)}_{n+2}=j+1|Y^{(k)}_n=t_j,Y^{(k)}_{n+1}=j)=P(Z_{n+2}=j|Z_{n+1}=j),
$$
and for all other values of $(u,t,s)$ we have 
$$
P(Y^{(k)}_{n+2}=u|Y^{(k)}_n=t,Y^{(k)}_{n+1}=s)=P(Z_{n+2}=u|Z_{n+1}=s).
$$    Thus for this process, all singletons $j$ for $j>k$ are
    still memory words of length one, but none of the $j$ with $0 \leq j \leq k$
    are.
    The main technical lemma that we will need is that the distribution
    of finite blocks up to some preassigned length $N$ of the  
$\{Y^{(k)}_n\}$ and  $\{Y^{(k+1)}_n\}$
    processes are arbitrarily close if the $t_{k +1}$ is chosen sufficiently
    large. This is independent of the other properties of $t_k$ that are needed,
    and so we begin by establishing this fact.

    Our proof will be via a coupling
    argument. We will calculate the finite distributions of these two processes
    by calculating time averages of a pair of typical sequences generated by
    transition matrices starting from the pair $00$. The coupling is especially
    easy since for both processes , from any state there is a fixed probability
    of moving to $00$ of at least $1 \over 8$ and therefore no matter how the
    two sequences diverge if we continue their evolution independently
    there is at every moment a fixed probability, namely
    $1 \over 64$, of the processes returning simultaneously to $00$.

\begin{lemma} With the definitions above for $\{Y^{(k)}_n\}$ and 
$\{Y^{(k+1)}_n\}$ if $N$ and
   $\delta > 0 $ are arbitrary, for any choice of $t_{k+1}$ that is sufficiently
   large we will have  that the variational distance between the
   distributions of 
$(Y^{(k)}_{0},\dots,Y^{(k)}_{N})$ and 
$(Y^{(k+1)}_{0},\dots,Y^{(k+1)}_{N})$ 
is at most $\delta$.
\end{lemma}

\begin{proof}
    Since the $\{Y^{(k)}_n\}$ is fixed at the start, given $N$ and
    $\delta$ we can choose $T$ sufficiently large so that the stationary
    probability $\pi^{(k)}(t) < \gamma$ for any $t>T$, where
    $\gamma = \delta/(64+2N)$.

   Suppose that we choose $t_{k+1} > T$. We begin the coupling by starting
   each of the processes at the pair $00$. Denote by $u_j$ and $v_j$ the
   random sequences constructed by applying the transition functions for the two
   processes $\{Y^{(k)}_n\} , \{Y^{(k+1)}_n\}$ respectively.
    Until $u_j = t_{k+1}$ for the first time the sequences
   can be taken to be identical,
   since they have the same transition probabilities for pairs
   that do not include this state. We denote by $\sigma_1$ this moment, and
   continue the coupling now independently waiting for the first moment
    $j > \sigma_1$ that  the equality $(u_j,u_{j+1}) = (v_j,v_{j+1}) = (0,0)$
   holds. Call this moment $\tau_1$.
   Notice that $\sigma_1$ is a function of the $u_j's$ while $\tau_1$
   is a function is a function of both processes.
   Beginning with $\tau_1$ we can once again continue
   the evolution in an identical fashion until the first moment $j > \tau_1$ that
     $u_j = t_{k+1}$. Call that stopping time $\sigma_2$. Note that this stopping time
     also depends on both processes.
   As before, as soon as this happens
     continue the processes independently until the first moment $j > \sigma_2$
     that the equality $(u_j,u_{j+1}) = (v_j,v_{j+1}) = (0,0)$ holds.
      It should now be clear how this is continued
     to build (with probability one) typical sequences for the two processes.
     In order to compare the stationary
     distributions of words up to length $N$ in the two processes we need to
     know what is the relative frequency of the periods when we are coupling independently
     compared to the periods when we are producing the same symbols.

    The asymptotic frequency of the occurrence of $t_{k+1}$ in the $u_j$ sequence
    is known to be at most $\gamma$ and at the stopping times $\sigma_i$
    $u_{\sigma_i} = t_{k+1}$.
     The gaps $\tau_i - \sigma_i$ are independent
    for different $i$'s and have a length which has a geometric distribution
    with fixed parameter $1 \over 64$ as we remarked earlier. Thus the average
    fraction of the time that the $N$-strings in the $u$ and $v$ sequences
     do not match exactly is at most $(64 + 2N) \gamma$. It follows
     that the variational distance between 
$(Y^{(k)}_{0},\dots,Y^{(k)}_{N} )$ and $(Y^{(k+1)}_{0},\dots,Y^{(k+1)}_{N} 
)$
  is at most $(64 + 2N) \gamma = \delta$
    and thus the the lemma has been established.

\end{proof}

  We can now give the

\bigskip
\noindent
{\bf Proof of Theorem~\ref{couplingtheorem}:}
\bigskip
\noindent

   Suppose that there does exist a sequence of stopping times and estimators
   as in the statement of the theorem. We begin with the one step Markov
   chain $Z_n$ described above and observe that the state $0$ has a positive
   stationary probability. Since the $\lambda_n$'s have  density one we can
   find an $N_0$ so that with probability at least $1 - {1 \over 10}$ in the
   string $Z^{N_0}_0$ there will be some $\lambda_n < N_0$ with $Z_{\lambda_n} = 0$
   and

   $$h_n(Z_{0},\dots,Z_{\lambda_n}) = 1 $$.

   We can  apply the lemma with $N=N_0$ and $\delta = {1 \over 10}$ to find
   a suitable $t_0$ with which we can define a $\{Y^{(0)}_n\}$ process in 
which now
   $0$ is not a memory word , so that for those strings where $Z_{\lambda_n} = 0$
    and $h_n(Z_{0},\dots,Z_{\lambda_n}) = 1 $ a definite mistake is being made.
    Such strings with length $N_0$ still have probability at least 
	$1-{2 \over 10}$.
   Having defined  $\{Y^{(0)}_n\}$ we notice now that the state $1$ is 
still a memory word
   of length one with positive stationary probability, and therefore we can find an
   $N_1$ sufficiently large so that with probability at least $1 - {1 \over 10^2}$ in the
   string $(Y^{(0)}_0,\dots,Y^{(0)}_{N_1})$ there will be some $\lambda_n 
< N_1$ with
   $Y^{(0)}_{\lambda_n} = 1$ and

   $$h_n(Y^{(0)}_{0},\dots,Y^{(0)}_{\lambda_n}) = 1 $$.

   As before we apply the lemma with $N=N_1$ and $\delta = 10^{-2}$ to find a
   suitable $t_1$ with which we can define the next process 
$\{Y^{(1)}_n\}$. For 
this process
   although $1$ fails to be a memory word we still can
   estimate  the probability that in a  
	$(Y^{(1)}_0,\dots,Y^{(1)}_{N_1})$ string there will be
   some $\lambda_n < N_1$ with $Y^{(1)}_{\lambda_n} = 1$
   and

   $$h_n(Y^{(1)}_{0},\dots,Y^{(1)}_{\lambda_n}) = 1 $$

   as being at least $1 - 2\times10^{-2}$. In addition, the previous estimate on
   strings of length $N_0$ is degraded only by $10^{-2}$, since the estimate on the
   variational distance descends to strings of shorter length.
   By now it should be clear how to continue the inductive construction of the $t_k$'s.
   The ultimate process that we obtain , which we may denote simply by 
	$\{X_n\}$ is
   of course a Markov chain of order two, and it has no memory words of length
   one at all. However, for every $k$, the probability that
   there will be
   some $\lambda_n < N_k$ with $X_{\lambda_n} = k$
   and

   $$h_n(X_{0},\dots,X_{\lambda_n}) = 1 $$

   will be at least $1 - {2 \over {9\times 10^k}}$. The Borel-Cantelli
   lemma implies that with probability
   one there will be infinitely many mistakes being made by our estimator
   contrary to the assumption. This concludes the proof of Theorem~\ref{couplingtheorem}.

\section{Limitations for Binary Finitarily Markovian Processes}

\bigskip
\noindent

       In the preceding section we showed that we cannot achieve
       density one in the forward memory length estimation problem
       even in the class of Markov chains on a countable alphabet.
       In this section we shall show something similar in the
       class of binary (i.e. ${0,1}$) valued finitarily Markov
       processes.
       To prove  this  we will assume
       that there is given a sequence of estimators and
       stopping times, $(h_n, \lambda_n)$
       that do succeed to estimate successfully the memory length
       for binary Markov chains of finite order and construct a
       finitarily Markovian binary process on which the scheme
       fails infinitely often. This differs from the proof outline
       of the previous section. There
       a contradiction was reached showing that the purported estimators
       do not exist.
       In the present case,
         as we remarked  in the
       opening paragraph of $\S 5$  there does exist a sequence of estimators
       $h_n$ which eventually succeed in giving the memory length
       almost surely for all binary Markov chains of finite order.
       Here is a precise statement:
\begin{theorem} \label{Thlimits}
 For any strictly increasing sequence of stopping times $\{\lambda_n\}$
and sequence of estimators
$\{h_n(X_{0},\dots,X_{\lambda_n})\}$,
such that for all stationary and ergodic binary
Markov chains with arbitrary finite order, $\lim_{n\to\infty} {\lambda_n\over n}=1$,
 and
$$ \lim_{n \rightarrow \infty} |h_n(X_{0},\dots,X_{\lambda_n}) - K(X_0^{\lambda_n})| =0
\ \ \mbox{almost surely}
$$
 there is a stationary, ergodic  finitarily Markovian binary time series
 such that on a set of positive measure of process realizations
$$
 h_n(X_0,\dots,X_{\lambda_n})\neq K(X^{\lambda_n}_{-\infty})
$$
infinitely often.
\end{theorem}

\proof

First  we  define the same
Markov-chain as in  Ryabko \cite{Ryabko88} (cf. also Gy\"{o}rfi, Morvai, Yakowitz
\cite{GYMY98}, Morvai and Weiss \cite{SPL-MW05})
 which serves as the technical tool for  construction of our
counterexample. Let the state space $S$ be the non-negative integers and define
the transition probabilities $p_{i,j}$ as follows:

$p_{0,1}=p_{1,2}=1 $, and for all $s>1$: $p_{s,0} = p_{s,s+1} = {1\over 2}$.

 This construction yields a stationary and ergodic
Markov chain  $\{M_i\}$ with stationary distribution
$$
P(M=0)=P(M=1)={1\over 4}
$$
and
$$
P(M=i)={1\over  2^{i}} \mbox{\ \  for $i\ge 2$}.
$$

  We shall construct a finitarily Markovian process $X_n$ by defining a certain  function
$f$ from the state space $S$ to $\{0,1\}$ and setting $X_n = f(M_n)$. We will ensure
that it is finitarily Markovian by taking care that $f(0)=f(1)=0 , f(2)=1$ and for
all $s>2$ if $f(s)=0$ then $f(s+1)=1$. Thus, in the $X_n$ process whenever one
observes two successive zeroes and a one, it is  known that the underlying states in the Markov
chain were $012$. Note that if
 there is  an integer $K$ such that
$f(i)=1$ for all $i\ge K-1$ then the  process $\{X_n\}$ is a binary
Markov-chain with order
not greater than $K$.
(Indeed, the probabilities $P(X_n=1|X_0, \dots, X_{n-1})$ are
determined by the last $K$ bits $(X_{n-K},\dots,X_{n-1})$.)

We will define $f$ in stages using the stopping times and estimators that the
hypotheses of the theorem give us. At  stage $j$ there will be an $f^{(j)}$ and
 we will define a binary-valued process, $\{X^{(j)}_i\}$
by the formula: $X^{(j)}_i=f^{(j)}(M_i)$ where $f^{(j)}$ will be a $\{0,1\}$ valued
function of the state
space $S$ which is eventually one. As remarked, this ensures that all these processes
are actually finite order Markov chains. The desired $f$ will be the limit of these
$f^{(j)}$'s and it will take the value $0$ infinitely often. Now for the definition.

For all $0\le j\le \infty$,  set $f^{(j)}(0)=0$,
$f^{(j)}(1)=0$, and $f^{(j)}(2)=1$.

Define $f^{(0)}(k)=1$ for all $k\ge 3$, hence   since $f^{(0)}(i)$ is
eventually $1$, the  process $\{X^{(0)}_i = f^{(0)}(M_i)\}$ is a stationary
ergodic binary Markov chain with
order $k_0\le 3$.

\noindent
  Recalling the stopping times and estimators define the event
\begin{eqnarray*}
A_1(t_1,s_1)&=& \{ \mbox{For some $n$: }    h_n(f^{(0)}(M_0),\dots,f^{(0)}(M_{\lambda_n}))\le k_0 , \\
&&  f^{(0)}(M_i)=1 \ \mbox{for $\lambda_n-k_0+1\le i\le \lambda_n$,}
t_1\le\lambda_n\le s_1.\}
\end{eqnarray*}

Notice that this is a well defined event in the sample space of the Markov chain
${M_n}$. All of the events that we are about to define are in that one fixed sample
space, only the function $f^{(i)}$ will be changing.
By the hypotheses of the theorem
there are sufficiently large $s_1>t_1>3$ such that the probability
$$
P(A_1(t_1,s_1)|M_0=0, M_1=1, M_2=2) > 1-2^{-1}.
$$
Let  $f^{(1)}(i)=f^{(0)}(i)$ for $i=0,1 \dots,s_1$ and let $f^{(1)}(s_1+1)=0$,
$f^{(1)}(i)=1$ for $i\ge s_1+2$. It is clear that the memory of a sequence
with prefix $1^0_{-k_0-1}$ in the process $X^{(1)}_n = f^{(1)}(M_n)$ is
greater then $k_0$ but if the event $A_1$ occurs then the estimator will commit an
error at least once in the
interval $[t_1,s_1]$. The new process has an order $k_1 \leq s_1 + 3$.
We will  continue in this manner inductively. Assuming that we have already defined
$k_j,t_j,s_j, A_j$ we will now go to stage $j+1$ and show how to update  these
parameters.

Let $s_{j+1}>t_{j+1}>s_j+3$ be chosen such that for the event
\begin{eqnarray*}
A_{j+1}(t_{j+1},s_{j+1})&=&
 \{\mbox{For some $n$ : }    
h_n(f^{(j)}(M_0),\dots,f^{(j)}(M_{\lambda_n}))\le k_j,\\
&&
f^{(j)}(M_i)=1 \ \mbox{for $\lambda_n-k_j+1\le i\le \lambda_n$,}
t_{j+1}\le\lambda_n\le s_{j+1} \}.
\end{eqnarray*}
we have:
$$
P(A_{j+1}(t_{j+1},s_{j+1})|M_0=M_1=0, M_2=1)>1-2^{-{(j+1)}}.
$$
Set now $f^{(j+1)}(i)=f^{(j)}(i)$ for $i=0,1,\dots,s_{j+1}$ and let
$f^{(j+1)}(s_{j+1}+1)=0$ and
$f^{(j+1)}(i)=1$ for $i\ge s_{j+1}+2$. It is clear that the memory of a
sequence with suffix a string of $1$'s of length $k_j$  in the process $X^{(j+1)}_n = f^{(j+1)}(M_n)$
is
greater then $k_j$ and if the event $A_{j+1}$ happens then the
estimator will commit error at least once in the
interval $[t_{j+1},s_{j+1}]$. The new process has  an order
$k_{j+1} \leq s_{j+1} + 3$.
By induction, we have defined all the functions
$f^{(j)}$ for $0\le j<\infty$.
To complete the definition of $f$ simply put $f=\lim f^{(j)}$. By the construction
this is certainly well defined.

By the Borel-Cantelli Lemma, conditioned on the positive probability 
event $M_0M_1M_2 = 012$,
the events $A_j$ occur infinitely often
almost surely and this completes
the proof of  Theorem~\ref{Thlimits}.

\begin{remark}

   In the final process $X_n$  that we constructed 
$P(K(X_{-\infty}^0) = k)$ decays to zero exponentially fast and
in particular is summable. It follows that with probability one
eventually $K(X_0^{n}) \leq n$ so that the reason for our failure
to estimate the memory length correctly is not coming about because we 
don't even
see the memory word.

   It is also worth pointing out the sequence of moments on which
   the estimator is failing is of density zero. It  follows
   fairly easily from the ergodic theorem that if one is willing to
   tolerate such failures then a straightforward application
   of any backward estimation scheme will converge outside a set
   of density zero. The effort that we expended in $\S 3,4$ to achieve
   density $1 -\epsilon$ for the stopping was because eventually we wanted to
   guarantee that there would be no failures at all.

   \end{remark}

\section{Forward Estimation of the Conditional Probability for Finitarily Markovian Processes}

Let  the alphabet be finite or countably infinite.
Now our goal is to estimate the conditional probability $P(X_{n+1}=x|X^n_0)$ on stopping times
in a pointwise sense.

\noindent
Let ${\cal N}$ be a sequence of stopping times such that eventually almost surely $X^n_{n-K(X^n_{-\infty})+1}$ appears
at least
$n^{1-\gamma}$ times in $X^n_0$.

\noindent
Let $\rho_n$ be any estimate of the length of the memory from samples $X^n_0$ such that
$\rho_n-K(X^n_{-\infty})\to 0$ on ${\cal N}$.

\noindent
Define our estimate ${\hat q}_n(x)$ of the conditional probability $P(X_{n+1}=x|X^n_0)$ on ${\cal N}$ as
$$
{\hat q}_n(x)= { \#\{ \rho_n-1\le  i<n : X^i_{i-\rho_n+1}= X^n_{n-\rho_n+1}, 
X_{i+1}=x \}
\over
\#\{ \rho_n-1\le  i<n : X^i_{i-\rho_n+1}= X^n_{n-\rho_n+1} \} }.
$$

\begin{theorem} \label{fmprobtheorem} On $n\in {\cal N}$,
$$
|{\hat q}_n(x)-P(X_{n+1}=x|X^n_0)|\to 0 \ \mbox{almost surely.}
$$
\end{theorem}

\begin{corollary}
For the stopping times ${\cal N}$  and estimator $\rho_n$ in Theorem~\ref{postheoremp}, Theorem~\ref{fmprobtheorem}
holds and
the density of ${\cal N}$ is at least $1-\epsilon$.
\end{corollary}

\noindent
To prove the above theorem we define Markov estimators of the conditional probabilities.
Use $\lambda_{n,K(X^n_{-\infty}),i}^-$ defined in (\ref{neglambda}).
Define the Markov estimator using $j$ samples as
\begin{equation}
q_n^j(x)={1\over j} \sum_{i=1}^j 1_{\{X_{n-\lambda_{n,K(X^n_{-\infty}),i}^-+1}=x\}}.
\end{equation}

\begin{lemma} \label{markestcon} Almost surely,
$$
\max_{j\ge \lfloor n^{1-\gamma}\rfloor } |q_n^j(x)-P(X_{n+1}=x|X^n_{-\infty})|\to 0.
$$
\end{lemma}
\proof
Since by Lemma ~\ref{iidlemma}, $q_n^j(x)$ is an average of independent and identically distributed bounded random
variables
so one may apply Hoeffding's inequality (cf. Hoeffding \cite{Hoeffding63} or Theorem 8.1 of Devroye et. al.
\cite{DGYL96}):
$$
\sum_{j= \lfloor n^{1-\gamma}\rfloor }^{\infty}
 P(|q_n^j(x)-P(X_{n+1}=x|X^n_{-\infty})|>\epsilon)\le
\sum_{j= \lfloor n^{1-\gamma}\rfloor }^{\infty}
2 e^{-2\epsilon^2 j}.
$$
The right hand side is summable in $n$ and the Borel-Cantelli lemma yields  Lemma~\ref{markestcon}. The proof of
Lemma~\ref{markestcon} is complete.

\bigskip
\noindent
{\bf Proof of Theorem~\ref{fmprobtheorem}:}
Since eventually, for $n\in {\cal N}$,  $X^n_{n-K(X^n_{-\infty})+1}$ appears $n^{1-\gamma}$ times in $X^n_0$ so
 ${\hat q}_n(x)=q_n^j(x)$ for some $j\ge \lfloor n^{1-\gamma}\rfloor$, and
$P(X_{n+1}=x|X^n_0)=P(X_{n+1}=x|X^n_{-\infty})$, the result follows from Lemma~\ref{markestcon}.
The proof of Theorem~\ref{fmprobtheorem} is complete.

\section{Forward Estimation of the Conditional Probability for  Markov Processes}

Let $\{X_n\}$ be a stationary and ergodic finite or countably infinite alphabet  Markov chain with order $K$.
Let $ORDEST_n$ be an estimator of the order from samples $X^n_0$ such that $ORDEST_n\to K$ almost surely.
Such an estimator can be found e.g. in Morvai and Weiss \cite{MW05}.
Let $ n\in {\cal N}$ if $X^n_{n-ORDEST_n+1}$ appears at least $n^{1-\gamma}$
times
in $X^n_0$. ${\cal N}$ is a sequence of stopping times.
Let
$$
{\hat q}_n(x)= 
{ \#\{ ORDEST_n-1\le  i<n : X^i_{i-ORDEST_n+1}= X^n_{n-ORDEST_n+1}, 
X_{i+1}=x \}
\over
\#\{ ORDEST_n-1\le  i<n : X^i_{i-ORDEST_n+1}= X^n_{n-ORDEST_n+1} \} }.
$$

\begin{theorem} \label{markovprobthm}
Assume $ORDEST_n$ equals the order eventually almost surely.Then
on $n\in {\cal N}$,
$$
|{\hat q}_n(x)-P(X_{n+1}=x|X^n_{n-K})|\to 0 \  \mbox{almost surely.}
$$
and
$$
\liminf_{n\to\infty} { \left| {\cal N}\bigcap \{0,1,\dots,n-1\} \right| \over n}=1.
$$
If the Markov chain turns out to take values from a finite set, then ${\cal N}$ takes as values  all
but finitely many positive integers.
\end{theorem}

\noindent
To prove the above theorem we define Markov estimators of the conditional probabilities.
Use $\lambda_{n,K,i}^-$ defined in (\ref{neglambda}).
Define the Markov estimator using $j$ samples as
\begin{equation}
q_n^j(x)={1\over j} \sum_{i=1}^j 1_{\{X_{n-\lambda_{n,K,i}^- +1}=x\}}.
\end{equation}

\begin{lemma} \label{ordmarkestcon} Almost surely,
$$
\max_{j\ge \lfloor n^{1-\gamma}\rfloor } |q_n^j(x)-P(X_{n+1}=x|X^n_{n-K})|\to 0.
$$
\end{lemma}
\proof
The proof goes along the lines of the proof of Lemma~\ref{markestcon}.
The proof of
Lemma~\ref{ordmarkestcon} is complete.

\noindent
{\bf Proof of Theorem~\ref{markovprobthm}:}
Since $ORDEST_n=K$ eventually, and so for $n\in {\cal N}$: $X^n_{n-K+1}$ appears at least $n^{1-\gamma}$ times thus
${\hat q}_n(x)=q_n^j(x)$ for some $j\ge \lfloor n^{1-\gamma}\rfloor$ and the result follows from Lemma~\ref{ordmarkestcon}.
Since any word of length $K$ with positive probability appears eventually almost surely $n^{1-\gamma}$ times in $X^n_0$
thus ${\cal N}$ has density one.
If the alphabet is finite, then the number of words with length $K$ is finite and by ergodicity, eventually almost
surely  all words with length $K$ which has positive probability appears at least $n^{1-\gamma}$ times.
The proof of Theorem~\ref{markovprobthm} is complete.



\end{document}